\documentclass[oneside,a4paper,english, 11pt]{amsart}

\usepackage{adjustbox}
\usepackage[latin9]{inputenc}
\usepackage{hyperref}
\usepackage{lscape}
\usepackage{slashbox}
\usepackage{mathrsfs}
\usepackage[title]{appendix}
\usepackage{enumerate}
\usepackage{enumitem}
\usepackage{xcolor}
\usepackage{bm}
\usepackage[sort=def]{glossaries}

\makeatletter

\numberwithin{equation}{section}

 \theoremstyle{plain}
\newtheorem{thm}[equation]{Theorem}
 \theoremstyle{plain}
\newtheorem{cond}[equation]{Condition}
 \theoremstyle{plain}

\theoremstyle{plain}

\theoremstyle{plain}

\theoremstyle{plain}

\theoremstyle{plain}

\theoremstyle{plain}
  \newtheorem{prop}[equation]{Proposition}
\theoremstyle{plain}
 \newtheorem{lemma}[equation]{Lemma}
\theoremstyle{plain}

\theoremstyle{plain}

\theoremstyle{plain}

\theoremstyle{definition}
  \newtheorem{defn}[equation]{Definition}
\theoremstyle{definition}
  
 \theoremstyle{definition}

\theoremstyle{remark}
\newtheorem{rmk}[equation]{Remark}

\usepackage{amsmath}
\usepackage{amssymb}
\usepackage{amscd}
\usepackage{amsthm}
\usepackage{array,multirow}
\usepackage[arrow,matrix,tips,all,cmtip,poly,color]{xy}

\usepackage{stmaryrd}
\usepackage{url}
\usepackage{tikz-cd}
\usepackage{color}
\usepackage{caption}
\usepackage{float}

\pdfpageheight\paperheight
\pdfpagewidth\paperwidth
\newcommand{\Z}{\mathbb{Z}}

\newcommand{\Q}{\mathbb{Q}}

\newcommand{\R}{\mathbb{R}}
\newcommand{\C}{\mathbb{C}}

\newcommand{\F}{\mathbb{F}}

\newcommand{\PP}{\mathbb{P}}

\newcommand{\A}{\mathbb{A}}

\newcommand{\fp}{\mathfrak{p}}
\newcommand{\fq}{\mathfrak{q}}

\newcommand{\fm}{\mathfrak{m}}

\newcommand{\cF}{\mathcal{F}}

\newcommand{\cP}{\mathcal{P}}

\newcommand{\Hom}{\mathrm{Hom}}

\newcommand{\End}{\mathrm{End}}

\newcommand{\GL}{\mathrm{GL}}

\newcommand{\tld}[1]{\widetilde{#1}}

\newcommand{\defeq}{\stackrel{\textrm{\tiny{def}}}{=}}

\title{A note on mod-$p$ local-global compatibility via Scholze's functor}
\author{Kegang Liu, Zicheng Qian}

\address{Universit\'e Sorbonne
Paris Nord, LAGA, 99 avenue J.B. Cl\'ement, 93430, Villetaneuse, France}
\email{kegang.liu@math.univ-paris13.fr}

\address{Department of Mathematics, University of Toronto, 40 St. George Street, Toronto, ON M5S 2E4,
Canada
}
\email{zqian@math.toronto.edu}

\begin{document}

\maketitle
\tableofcontents

\section{Introduction}\label{sec: intro}
Let $L$ be a $p$-adic field and $\mathrm{Gal}_L$ be its absolute Galois group. In \cite{Sch18}, Scholze constructs a functor which sends an admissible $\Z_p$-representation $\pi$ of $\mathrm{GL}_n(L)$ to certain cohomology group $H^{i}_{\text{\'et}}(\PP_{\C_p}^{n-1}, \cF_\pi)$ (for $i\geq 0$) equipped with the action of $D^\times\times\mathrm{Gal}_L$ with $D$ the division algebra over $L$ with invariant $\frac{1}{n}$. Then Scholze proves a version of local-global compatibility by applying his functor to a space of $\Q_p/\Z_p$-automorphic forms on a certain inner form of $\GL_2$ over some totally real field. In \cite{Liu21}, Liu considers a space $\pi$ of $\Q_p/\Z_p$-automorphic forms (with arbitrary level at some $\fp\mid p$) on certain unitary similitude group over a totally real field $F^+$ which splits at $\fp$, and then relates it to the cohomology of certain Kottwitz-Harris-Taylor type Shimura varieties via Scholze's functor. Given a Hecke system $\fm$ which appears only in the middle degree of the cohomology of Shimura varieties and assume that it corresponds to an absolutely irreducible residual Galois representation $\overline{\sigma}_{\fm}$, Liu proves that $H^{n-1}_{\text{\'et}}(\PP_{\C_p}^{n-1}, \cF_{\pi_{\fm}})$ is $\sigma_{\fm}|_{\mathrm{Gal}_{F^+_{\fp}}}$-typic where $\sigma_{\fm}$ is the lift of $\overline{\sigma}_{\fm}$ established in \cite[Proposition~2.3]{Liu21}. Furthermore, Liu proves that $H^{n-1}_{\text{\'et}}(\PP_{\C_p}^{n-1}, \cF_{\pi_{\fm}[\fm]})$ is sufficient to determine the restriction $\overline{\sigma}_{\fm}|_{\mathrm{Gal}_{F^+_{\fp}}}$ assuming that
\begin{itemize}
\item the dual $\pi_{\fm}^\vee$ is flat as a module over the big Hecke algebra;
\item $\overline{\sigma}_{\fm}|_{\mathrm{Gal}_{F^+_{\fp}}}$ is semisimple and multiplicity free.
\end{itemize}
In this short note, we completely remove the semisimple assumption in Liu's result. Namely, assuming flatness of $\pi_{\fm}^\vee$ and $\overline{\sigma}_{\fm}|_{\mathrm{Gal}_{F^+_{\fp}}}$ being multiplicity free, we prove that $H^{n-1}_{\text{\'et}}(\PP_{\C_p}^{n-1}, \cF_{\pi_{\fm}[\fm]})$ determines $\overline{\sigma}_{\fm}|_{\mathrm{Gal}_{F^+_{\fp}}}$ uniquely. We give remarks on our assumptions at the end of this note.

\medskip

\textbf{Acknowledgements}.  \,
We would like to thank Florian Herzig for reading an early draft of this note and pointing out one mistake, which leads to further simplification of our proof. The second author would like to thank Yongquan Hu for his feedback and for mentioning an example in \cite{HW21} (see Remark~\ref{rmk: non flat}).
\section{Submodules of a typic module}\label{sec: submodule}
Let $G$ be a group, $R$ be a commutative ring and $\rho_0$ be an $R[G]$-module of finite length. We have the following definition generalizing \cite[Definition~5.2]{Sch18}.
\begin{defn}\label{def: general typic}
Assume that $\rho_0$ is multiplicity free, namely each Jordan--H\"older factor of $\rho_0$ appears with multiplicity one. Then $\rho_0$ admits a decomposition $\rho_0\cong \bigoplus \tld{\rho}$ into its non-zero indecomposable direct summands. We say that an $R[G]$-module $V$ is \emph{$\rho_0$-typic} if there exists a non-zero $R$-module $W_{\tld{\rho}}$ with trivial $G$-action for each $\tld{\rho}$, such that $V\cong \bigoplus_{\tld{\rho}}W_{\tld{\rho}}\otimes_R \tld{\rho}$.
\end{defn}
We assume throughout this section that $R$ is field and $\rho_0$ is multiplicity free. We fix from now a $\rho_0$-typic $R[G]$-module $V$ equipped with an $R$-module $W_{\tld{\rho}}$ for each indecomposable direct summand $\tld{\rho}$ of $\rho_0$ as in Definition~\ref{def: general typic}. The main result of this section is a criterion (see Proposition~\ref{prop: main criterion}) for certain submodule of $V$ to be $\rho_0$-typic.

We write $\Sigma$ for the set of non-zero indecomposable $R[G]$-submodules of $\rho_0$, equipped with the natural partial order given by inclusion of $R[G]$-submodules. We write $\mathrm{JH}_{R[G]}(\cdot)$ for the set of Jordan--H\"older factors. As $\rho_0$ is multiplicity free, any $R[G]$-submodule of $\rho_0$ is uniquely determined by its set of Jordan--H\"older factors, and we clearly have $\#\Sigma\leq 2^\ell$ where $\ell$ is the length of $\rho_0$. Note that $V\cong \bigoplus_{\tld{\rho}}W_{\tld{\rho}}\otimes_R \tld{\rho}$ forces $V$ to be locally finite, and so is any subquotient of $V$.

\begin{lemma}\label{lem: extend}
Let $\rho'\subseteq \rho$ be two elements of $\Sigma$. Then the induced map $\Hom_{R[G]}(\rho,V)\rightarrow\Hom_{R[G]}(\rho',V)$ is an isomorphism.
\end{lemma}
\begin{proof}
We first deduce from $\rho,\rho'\in\Sigma$ and $\rho'\subseteq \rho$ that there exists a unique indecomposable direct summand $\tld{\rho}\in\Sigma$ of $\rho_0$ which contains $\rho,\rho'$. The canonical map $\Hom_{R[G]}(\rho,\tld{\rho})\rightarrow\Hom_{R[G]}(\rho',\tld{\rho})$ is clearly an isomorphism of $R$-vector spaces of dimension one. Then the canonical map in question factors through the isomorphisms
$$\Hom_{R[G]}(\rho,V)\cong W_{\tld{\rho}}\otimes_{R}\Hom_{R[G]}(\rho,\tld{\rho})\xrightarrow{\sim} W_{\tld{\rho}}\otimes_{R}\Hom_{R[G]}(\rho',\tld{\rho})\cong \Hom_{R[G]}(\rho',V).$$
\end{proof}

\begin{lemma}\label{lem: irr cosocle}
Let $V'\subseteq V$ be an $R[G]$-submodule with $\mathrm{cosoc}_{R[G]}(V')$ being irreducible. Then there exists $\rho\in\Sigma$ such that $V'\cong\rho$.
\end{lemma}
\begin{proof}
We write $\tau\defeq \mathrm{cosoc}_{R[G]}(V')\in\mathrm{JH}_{R[G]}(V)=\mathrm{JH}_{R[G]}(\rho_0)$. There exists a unique $\rho\subseteq \tld{\rho}\subseteq \rho_0$ such that $\tld{\rho}$ is an indecomposable direct summand of $\rho_0$ and $\mathrm{cosoc}_{R[G]}(\rho)\cong \tau$. As $V/W_{\tld{\rho}}\otimes_R\rho$ does not have $\tau$ as Jordan--H\"older factor, we may assume without loss of generality that $\rho=\tld{\rho}=\rho_0$. Then the key observation is that
\begin{equation}\label{equ: canonical lift}
\Hom_{R[G]}(\tld{\rho},V)\cong W_{\tld{\rho}}\otimes_{R}\End_{R[G]}(\tld{\rho})\xrightarrow{\sim}W_{\tld{\rho}}\otimes_{R}\End_{R[G]}(\tau)\cong \Hom_{R[G]}(\tau,W_{\tld{\rho}}\otimes_R\tau).
\end{equation}
The $R[G]$-submodule $V'\subseteq V$ determines an embedding $\tau\hookrightarrow W_{\tld{\rho}}\otimes_R\tau$ and thus (by (\ref{equ: canonical lift}) an embedding $f:\tld{\rho}\hookrightarrow V$. We write $\mathrm{rad}(V')$ for the kernel of $V'\twoheadrightarrow\mathrm{cosoc}_{R[G]}(V')$. As the canonical map $V'/\mathrm{rad}(V')\rightarrow V/(\mathrm{im}(f)+\mathrm{rad}(V'))$ is zero by the choice of $f$, so is the map $V'\rightarrow V/\mathrm{im}(f)$, which implies that $V'\subseteq \mathrm{im}(f)$. This inclusion must be an equality as both $R[G]$-modules share the same cosocle.
\end{proof}

\begin{lemma}\label{lem: mult free sub}
Let $V'\subseteq V$ be an $R[G]$-submodule. If $V'$ is multiplicity free, then there exists an embedding $V'\hookrightarrow\rho_0$.
\end{lemma}
\begin{proof}
By writing $V'$ as direct sum of its indecomposable direct summands, it suffices to assume that $V'$ is indecomposable and find $\rho\in\Sigma$ such that $V'\cong \rho$. As $\mathrm{JH}_{R[G]}(V)=\mathrm{JH}_{R[G]}(\rho_0)$ is finite, we deduce that $V'$ has finite length. By writing each $W_{\tld{\rho}}=\displaystyle{\varinjlim_k}W_{\tld{\rho},k}$ as a direct limit of its finite dimensional subspaces and then using the fact that $V'$ has finite length, we may assume without loss of generality that $W_{\tld{\rho}}$ is finite dimensional for each indecomposable direct summand $\tld{\rho}$ of $\rho_0$. We write $\mathrm{soc}_{R[G]}V'\cong\bigoplus_{t=1}^s\tau_t$, then each $\tau_t\subseteq V'\subseteq V$ determines a unique $\tld{\rho}_t$ containing $\tau_t$ as well as an element $f_t\in \Hom_{R[G]}(\tau_t,V)\cong \Hom_{R[G]}(\tld{\rho}_t,V)\cong W_{\tld{\rho}_t}$. As $V'$ is indecomposable and $W_{\tld{\rho}}\otimes_R\tld{\rho}$ do not share common Jordan--H\"older factor for different $\tld{\rho}$, we deduce that all $\tld{\rho}_t$ equal the same $\tld{\rho}$. As it is harmless to replace $R$ with its algebraic closure which is an infinite field, there exists $\ell: W_{\tld{\rho}}\rightarrow R$ such that $\ell(f_t)\neq 0$ for each $1\leq t\leq s$. Hence, $\ell\otimes_R\tld{\rho}: W_{\tld{\rho}}\otimes_R\tld{\rho}\rightarrow\tld{\rho}$ restricted to an injection on $\mathrm{soc}_{R[G]}(V')$, and thus an injection on $V'$ as well. We conclude by the observation that any indecomposable $R[G]$-submodule of $\tld{\rho}$ is in $\Sigma$.
\end{proof}

\begin{lemma}\label{lem: direct summand criterion}
Let $V'\subseteq V$ be an $R[G]$-submodule. Assume that
\begin{itemize}
\item $\mathrm{JH}_{R[G]}(V')=\mathrm{JH}_{R[G]}(\rho_0)$; and
\item for each indecomposable direct summand $\tld{\rho}$ of $\rho_0$ and each embedding $f: \tld{\rho}\hookrightarrow V$, we have either $\mathrm{im}(f)\subseteq V'$ or $\mathrm{im}(f)\cap V'=0$.
\end{itemize}
Then $V'$ is $\rho_0$-typic.
\end{lemma}
\begin{proof}
Recall that we have $V\cong \bigoplus_{\tld{\rho}}W_{\tld{\rho}}\otimes_R \tld{\rho}$ and the identification $W_{\tld{\rho}}\cong \Hom_{R[G]}(\tld{\rho},V)$ for each indecomposable direct summand $\tld{\rho}$ of $\rho_0$. We write $W'_{\tld{\rho}}\subseteq W_{\tld{\rho}}$ for the subspace of all morphisms $f: \tld{\rho}\rightarrow V$ satisfying $\mathrm{im}(f)\subseteq V'$. We claim that the natural map
\begin{equation}\label{equ: natural map}
\bigoplus_{\tld{\rho}}W'_{\tld{\rho}}\otimes_R \tld{\rho}\rightarrow V'
\end{equation}
is an isomorphism. The compatibility with $V\cong \bigoplus_{\tld{\rho}}W_{\tld{\rho}}\otimes_R \tld{\rho}$ forces (\ref{equ: natural map}) to be injective. As $V'$ is sum of its $R[G]$-submodules with irreducible cosocle, it suffices to prove that each such $R[G]$-submodule $V''$ of $V'$ is contained in the image of (\ref{equ: natural map}). In fact, it follows from Lemma~\ref{lem: irr cosocle} that there exists $\rho\in\Sigma$ such that $V''\cong \rho$. Hence, we deduce from Lemma~\ref{lem: extend} that there exists an indecomposable direct summand $\tld{\rho}$ of $\rho_0$ as well as $f\in\Hom_{R[G]}(\tld{\rho},V)$ such that $\tld{\rho}\supseteq \rho$ and $V''\subseteq \mathrm{im}(f)$. As $\mathrm{im}(f)$ is multiplicity free, it embeds into $\rho_0$ by Lemma~\ref{lem: mult free sub}, and thus embeds into $\tld{\rho}$ by checking Jordan--H\"older factors. This forces $\tld{\rho}\cong \mathrm{ker}(f)\oplus\mathrm{im}(f)$ and thus $\mathrm{ker}(f)=0$ as $\tld{\rho}$ is indecomposable. In other words, $f$ is an embedding with $0\neq V''\subseteq \mathrm{im}(f)\cap V'$, which together with our assumption implies that $\mathrm{im}(f)\subseteq V'$. Hence, $\mathrm{im}(f)$ is contained in the image of (\ref{equ: natural map}), and so is $V''$. Note that $\mathrm{JH}_{R[G]}(V')=\mathrm{JH}_{R[G]}(\rho_0)$ forces $W'_{\tld{\rho}}\neq 0$ for each indecomposable direct summand $\tld{\rho}$ of $\rho_0$. The proof is thus completed.
\end{proof}

\begin{prop}\label{prop: main criterion}
Let $\rho_0$ be a multiplicity free $R[G]$-module of finite length. Let $V$ be a $\rho_0$-typic $R[G]$-module with a sequence of sub $R[G]$-modules $V_1\subseteq V_2\subseteq \cdots$ satisfying the following conditions
\begin{itemize}
\item $V=\bigcup_{r\geq 1}V_r$; and
\item for each $r\geq 1$, there exists an embedding $V_{r+1}/V_r\hookrightarrow V_1^{\oplus s_r}$ for some $s_r\geq 1$.
\end{itemize}
Then $V_1$ is $\rho_0$-typic. In particular, $V_1$ determines $\rho_0$ up to isomorphism.
\end{prop}
\begin{proof}
Our assumption clearly implies that $\mathrm{JH}_{R[G]}(V_1)=\mathrm{JH}_{R[G]}(\rho_0)$. Let $\tld{\rho}$ be an indecomposable direct summand of $\rho_0$ and $f: \tld{\rho}\hookrightarrow V$ be an embedding. According to Lemma~\ref{lem: direct summand criterion}, it suffices to show that either $\mathrm{im}(f)\subseteq V_1$ or $\mathrm{im}(f)\cap V_1=0$. We set $V_{f,0}\defeq 0\subseteq \mathrm{im}(f)$ and $V_{f,r}\defeq \mathrm{im}(f)\cap V_r$ for each $r\geq 1$. Our assumption on $\{V_r\}_{r\geq 1}$ implies that $\{V_{f,r}\}_{r\geq 0}$ is an increasing and exhaustive filtration on $\mathrm{im}(f)$. The inclusion $\mathrm{im}(f)\subseteq V$ induces a natural embedding
$$V_{f,r+1}/V_{f,r}\hookrightarrow V_{r+1}/V_r\hookrightarrow V_1^{\oplus s_r}\hookrightarrow V^{\oplus s_r}.$$
As $V^{\oplus s_r}$ is $\rho_0$-typic and $V_{f,r+1}/V_{f,r}$ is multiplicity free, we deduce from Lemma~\ref{lem: mult free sub} that $V_{f,r+1}/V_{f,r}$ embeds into $\rho_0$, and actually embeds into $\tld{\rho}$ by checking Jordan--H\"older factors. As $V_{f,r+1}/V_{f,r}$ embeds into $\tld{\rho}\cong \mathrm{im}(f)$ for each $r\geq 0$, we deduce that
$$\tld{\rho}\cong \mathrm{im}(f)\cong \bigoplus_{r\geq 0}V_{r+1,f}/V_{r,f}.$$
However, $\tld{\rho}$ is indecomposable, and thus there exists a unique $r_f\geq 0$ such that $V_{r_f+1,f}/V_{r_f,f}\cong \tld{\rho}$ and $V_{r+1,f}=V_{r,f}$ for all $r\neq r_f$. In particular, we have $\mathrm{im}(f)\subseteq V_1$ if $r_f=0$, and $\mathrm{im}(f)\cap V_1=0$ if $r_f\geq 1$. As the $\rho_0$-typic $R[G]$-module $V_1$ determines the isomorphism class of each indecomposable direct summand $\tld{\rho}$ of $\rho_0$ (by considering all possible indecomposable direct summands of $V_1$), it clearly determines $\rho_0$ up to isomorphism. The proof is thus finished.
\end{proof}

We also have the following more general result on capturing $\rho_0$ from an $R[G]$-submodule $V'\subseteq V$ without knowing that $V'$ is $\rho_0$-typic.
\begin{prop}\label{prop: general capture}
Let $V'\subseteq V$ be an $R[G]$-submodule. Assume that $\mathrm{JH}_{R[G]}(V')=\mathrm{JH}_{R[G]}(\rho_0)$. Then $V'$ determines $\rho_0$ up to isomorphism.
\end{prop}
\begin{proof}
As $\rho_0$ is multiplicity free, for each $\tau\in\mathrm{JH}_{R[G]}(\rho_0)$, there exists a unique $R[G]$-submodule $\rho_\tau\subseteq \rho_0$ such that $\mathrm{cosoc}_{R[G]}(\rho_\tau)\cong \tau$. It follows from Lemma~\ref{lem: irr cosocle} that, for each $\tau\in\mathrm{JH}_{R[G]}(\rho_0)$, any $R[G]$-submodule $V''\subseteq V'$ satisfying $\mathrm{cosoc}_{R[G]}(V'')\cong \tau$ must also satisfy $V''\cong \rho_\tau$. Consequently, we deduce from $\mathrm{JH}_{R[G]}(V')=\mathrm{JH}_{R[G]}(\rho_0)$ that $V'$ determines the set of isomorphism classes $\{[\rho_\tau]\}_{\tau\in\mathrm{JH}_{R[G]}(\rho_0)}$. It suffices to show that $\{[\rho_\tau]\}_{\tau\in\mathrm{JH}_{R[G]}(\rho_0)}$ determines $\rho_0$ up to isomorphism. We prove that $\{[\rho_\tau]\}_{\tau\in\mathrm{JH}_{R[G]}(\rho_0)}$ determines $\rho$ up to isomorphism for each $R[G]$-submodule $\rho\subseteq \rho_0$ by induction on the length of $\rho$. Let $\rho'\subseteq \rho$ be two $R[G]$-submodules of $\rho_0$ with $\rho/\rho'\cong \tau_0$ for some $\tau_0\in\mathrm{JH}_{R[G]}(\rho_0)$. Assume first that $\{[\rho_\tau]\}_{\tau\in\mathrm{JH}_{R[G]}(\rho_0)}$ determines $\rho'$ and $\rho'\cap\rho_{\tau_0}$ up to isomorphism. We choose two embeddings $f_1: \rho'\cap \rho_{\tau_0}\rightarrow \rho'$ and $f_2: \rho'\cap\rho_{\tau_0}\rightarrow\rho_{\tau_0}$ and note that the choice of the pair $f_1,f_2$ is unique up to automorphims of $\rho'$, $\rho_{\tau_0}$ and $\rho'\cap\rho_{\tau_0}$. Hence, the isomorphism class of the amalgamate sum $\rho'\oplus_{\rho'\cap\rho_{\tau_0}}\rho_{\tau_0}$ does not depend on the choice of $f_1,f_2$. It is obvious that $\rho\cong \rho'\oplus_{\rho'\cap\rho_{\tau_0}}\rho_{\tau_0}$ and thus $\{[\rho_\tau]\}_{\tau\in\mathrm{JH}_{R[G]}(\rho_0)}$ determines $\rho$ up to isomorphism. The proof is thus finished by an induction on length.
\end{proof}

\section{Application to mod-$p$ local global compatibility via Scholze's functor}\label{sec: LGC}
We first establish the global setup for the study of the cohomology of the relevant Shimura varieties.

Fix an integer $n>2$ and a prime $p$. Let $F$ be a CM field which is a quadratic extension of its maximal totally real subfield $F^+$. We write $c$ for the unique non-trivial element of $\mathrm{Gal}(F/F^+)$. Let $B$ be the division algebra over $F$ of dimension $n^2$ as chosen in Section 0.1 of \cite{BZ99} (thus equipped with a certain involution $b \mapsto b^*$ on it) and
 $\tld{G}$ be the algebraic group over $F^+$ whose group of $R$-points for any $F^+$-algebra $R$ is given by
$$\tld{G}(R)\defeq \{(g,\lambda)\in(B^{\text{op}}\otimes_{F^+}R)^\times\times R^\times |\,gg^*=\lambda\}.$$
We assume that $v$ splits in $F$ for each finite place $v$ of $F^+$ dividing $p$. We fix a finite place $\fp$ of $F^+$ (resp.~$\fq$ of $F$) such that $\fp=\fq\fq^c$. Then our choice of the division algebra $B$ above implies that $B_{\fq}$ is a division algebra over $F_{\fq}$ of invariant $\frac{1}{n}$.

Let $G\defeq \mathrm{Res}_{F^+/\Q}(\tld{G})$ be the Weil restriction of scalars. Let $\mathbb{S}\defeq \mathrm{Res}_{\C/\R}(\mathbb{G}_m)$ be the Deligne torus and $h$ be a morphism
$$h: \mathbb{S} \to G_{\R}$$
such that $h$ defines on $W_{\R}$ a Hodge structure of type $(1,0), (0,1)$ and such that  $\psi(w_1,h(i)w_2)$ is a symmetric positive definite bilinear form on $W_{\R}$. Note that $h$ is unique up to $G(\mathbb{R)}$-conjugacy and we let $X$ denote the $G(\mathbb{R)}$-conjugacy class of $h$. Then $(G,X)$ defines a Shimura datum and for sufficiently small compact open subgroups $U \subseteq G(\A_{\Q}^{\infty})$ we have a projective system of Shimura varieties $\mathrm{Sh}_U$ over its reflex field, which can be identified with $F$ in a canonical way. We will write $\mathrm{Sh}_{KU^{\fp}}$ instead of $\mathrm{Sh}_{K\times (F^+_{\fp})^\times \times U^{\fp}}$ for the Shimura variety associated with $U=K\times (F^+_{\fp})^\times \times U^{\fp}$ for $K \subseteq (B_{\fq}^{\text{op}})^\times$ compact open and $U^{\fp} \subseteq \tld{G}(\A_{F^+}^{\infty,\fp})$ (sufficiently small) compact open.

We fix a tame level, i.e. a compact open subgroup $U^{\fp}=\prod_{v\neq \fp}U_v$ of $\tld{G}(\A_{F^+}^{\infty,\fp})$ and let $\cP$ denote the set of finite places $v$ of $F^+$ such that
\begin{itemize}
    \item $v\nmid p$;
    \item $v$ splits in $F$;
    \item $\tld{G}(F^+_v) \cong \GL_n(F^+_v) \times (F^+_v)^\times $ and $U_v$ is a maximal compact open subgroup of $\tld{G}(F^+_v)$.
\end{itemize}
Consider the abstract Hecke algebra
$$\mathbb{T}_{\cP}\defeq  \Z[T^{(j)}_{w},~T^{(j)}_{w^c}: v=ww^c\in \cP, 1\leq j\leq n]$$
where $T^{(j)}_{w}$ is the Hecke operator corresponding to the double coset
$$ \Big[\GL_{n}(\mathcal{O}_{F_{w}})
\begin{pmatrix}
\varpi_{w}1_{j} & 0\\
0 & 1_{n-j}
\end{pmatrix}\GL_{n}(\mathcal{O}_{F_{w}}) \Big].$$
Here $\varpi_{w}$ is a uniformizer of the local field $F_w$.
Then the Hecke algebra $\mathbb{T}_{\cP}$ acts on $H^i(\mathrm{Sh}_{KU^{\fp},\C},\Z)$ for all compact open $K\subseteq (B_{\fq}^{\text{op}})^\times$.

Let $\F$ be a finite extension of $\F_p$ and $\overline{\sigma}: \mathrm{Gal}_F \to \GL_{n}(\F)$ an $n$-dimensional absolutely irreducible Galois representation which is unramified at each place of $F$ dividing some $v\in\cP$. Hence, we can associate a maximal ideal $\fm\subseteq \mathbb{T}_{\cP}$ with $\overline{\sigma}$ (cf.~the paragraph before Condition~2.1 of \cite{Liu21}).

We assume the following condition from now on:
\begin{cond}\label{cond: middle deg}
For each $K\subseteq (B_{\fq}^{\emph{op}})^\times$ compact open,
$$H^{i}(\mathrm{Sh}_{KU^{\fp},\C},\Z)_{\fm} \neq 0$$
only if $i = n-1$.
\end{cond}
Let $\mathbb{T}(KU^{\fp})$ be the image of $\mathbb{T}_{\cP}$ in $\mathrm{End}(H^{n-1}(\mathrm{Sh}_{KU^{\fp},\C},\Z))$ and $\mathbb{T}(KU^{\fp})_{\fm}$ be its $\fm$-adic completion. We also consider the big Hecke algebra
$$\mathbb{T}(U^{\fp})_{\fm}\defeq \underset{U}{\underleftarrow{\mathrm{lim}}}\mathbb{T}(KU^{\fp})_{\fm}$$
which is a complete Noetherian local ring with finite residue field. 
Let $\sigma: \mathrm{Gal}_F \to \GL_{n}(\mathbb{T}(U^{\fp})_{\fm})$ be the unique (up to conjugation) lift of $\overline{\sigma}$ characterized by \cite[Proposition 2.3]{Liu21}.

Let $G'$ be the inner form of $G$ over $F^+$ such that $G'(F^+ \otimes_{\Q}\R)$ is compact modulo center, $G'(\A_{F^+}^{\infty,\fp}) = \tld{G}(\A_{F^+}^{\infty,\fp})$, and $G'(F^+_{\fp}) \cong \GL_n(F^+_{\fp}) \times (F^+_{\fp})^\times$.
Let
$\pi_{U^{\fp}}$ be the admissible $\Z_p$-representation of $\GL_n(F^+_{\fp})$ given by the space of continuous functions
$$\pi_{U^{\fp}}\defeq C^0(G'(F^+)\backslash G'(\A_{F^+}^{\infty})/((F^+_{\fp})^\times \times U^{\fp}), \Q_p/\Z_p).$$
By Corollary 6.7 of \cite{Liu21} the natural action of $\mathbb{T}_{\cP}$ on
$$\pi_{\fm}\defeq \pi_{{U^{\fp}},\fm}=C^0(G'(F^+)\backslash G'(\A_{F^+}^{\infty})/((F^+_{\fp})^\times \times U^{\fp}), \Q_p/\Z_p)_{\fm}$$
extends to a continuous action of $\mathbb{T}(U^{\fp})_{\fm}$.

In \cite[Section~3,4]{Sch18}, for each $p$-adic field $L$ and $i\geq 0$, Scholze defines a functor which sends an admissible smooth $\Z_p[\GL_n(L)]$-module $\pi$ (cf.~\cite[Definition~4.1]{Sch18}) to
$$H^{i}_{\text{\'et}}(\PP_{\C_p}^{n-1}, \cF_\pi)$$
with a natural action by $D^\times\times\mathrm{Gal}_{L}$ (see \cite[Proposition~3.1]{Sch18} for the definition of the sheaf $\cF_\pi$). Here $D$ is the central division algebra over $L$ of invariant $\frac{1}{n}$. Although not used in the rest of this note, we remark that the $D^\times$-representation $H^{i}_{\text{\'et}}(\PP_{\C_p}^{n-1}, \cF_\pi)$ is known to be admissible by \cite[Theorem~4.4]{Sch18}.

We have the following typicity result from \cite[Corollary~7.1]{Liu21}:
\begin{prop}
Assume that Condition~\ref{cond: middle deg} holds for $\fm$. Then $H^{n-1}_{\emph{\'et}}(\PP_{\C_p}^{n-1}, \cF_{\pi_{\fm}})$ is a $\sigma|_{\mathrm{Gal}_{F^+_{\fp}}}$-typic $\mathbb{T}(U^{\fp})_{\fm}[\mathrm{Gal}_{F^+_{\fp}}]$-module. In particular, $H^{n-1}_{\emph{\'et}}(\PP_{\C_p}^{n-1}, \cF_{\pi_{\fm}})[\fm]$ is a $\overline{\sigma}|_{\mathrm{Gal}_{F^+_{\fp}} }$-typic $\F[\mathrm{Gal}_{F^+_{\fp}}]$-module.
\end{prop}


Now we need another condition to apply various results we need from \cite[Section~7]{Liu21}.
\begin{cond}\label{cond: main}
The dual $\pi_{\fm}^\vee = \mathrm{Hom}_{\Z_p}(\pi_{\fm}, \Q_p/\Z_p) $ is flat as a module over $\mathbb{T}(U^{\fp})_{\fm}$.
\end{cond}
Under the Condition~\ref{cond: main}, there exists for each $r \geq 1$ a short exact sequence (see \cite[Lemma~7.5]{Liu21})
\begin{equation}\label{equ: ses1}
0 \rightarrow \pi_{\fm}[\fm^r] \rightarrow \pi_{\fm}[\fm^{r+1}] \rightarrow  (\pi_{\fm}[\fm])^{\oplus s_r} \rightarrow 0
\end{equation}
where $s_r \geq 1$ is a positive integer. Applying the functor $H^{n-1}_{\text{\'et}}(\PP_{\C_p}^{n-1}, \cF_{-})$, we obtain an exact sequence on cohomology groups (for each $r \geq 1$)
\begin{equation}\label{equ: ses2}
0 \rightarrow H^{n-1}_{\text{\'et}}(\PP_{\C_p}^{n-1}, \cF_{\pi_{\fm}[\fm^r]}) \rightarrow H^{n-1}_{\text{\'et}}(\PP_{\C_p}^{n-1}, \cF_{\pi_{\fm}[\fm^{r+1}]})   \rightarrow \bigoplus_{s_r}H^{n-1}_{\text{\'et}}(\PP_{\C_p}^{n-1}, \cF_{\pi_{\fm}[\fm]}).
\end{equation}
The injectivity on the left hand side of (\ref{equ: ses2}) follows from \cite[Lemma~7.9]{Liu21}.
Now we set
$$V_r\defeq H^{n-1}_{\text{\'et}}(\PP_{\C_p}^{n-1}, \cF_{\pi_{\fm}[\fm^r]})[\fm]$$
for each $r\geq 1$ and note that $V_1=H^{n-1}_{\text{\'et}}(\PP_{\C_p}^{n-1}, \cF_{\pi_{\fm}[\fm]})$.
Taking $\fm$-torsion on the sequence \eqref{equ: ses2} yields
\begin{equation}\label{equ: ses3}
0 \rightarrow V_r \to V_{r+1} \rightarrow  (V_1)^{\oplus s_r}.
\end{equation}
To apply our results in Section~\ref{sec: submodule}, we further assume that
\begin{cond}\label{cond: mult free}
The $\F[\mathrm{Gal}_{F^+_{\fp}}]$-module $\overline{\sigma}|_{\mathrm{Gal}_{F^+_{\fp}}}$ is multiplicity free.
\end{cond}
Then we take
\begin{itemize}
    \item $R=\F$;
    \item $G=\mathrm{Gal}_{F_{\fp}^+}$;
    \item $\rho_0=\overline{\sigma}|_{\mathrm{Gal}_{F^+_{\fp}}}$; and
    \item $V=H^{n-1}_{\text{\'et}}(\PP_{\C_p}^{n-1}, \cF_{\pi_{\fm}})[\fm]$.
\end{itemize}
It follows from \cite[Lemma~7.7]{Liu21} that $V = \displaystyle{\varinjlim_r}V_r$, which together with (\ref{equ: ses3}) fulfills all the conditions of Proposition~\ref{prop: main criterion}. Therefore we deduce:
\begin{thm}\label{thm: main}
Assume that Condition~\ref{cond: middle deg}, Condition~\ref{cond: main} and Condition~\ref{cond: mult free} hold for $\fm$. Then
\begin{itemize}
\item $H^{n-1}_{\emph{\'et}}(\PP_{\C_p}^{n-1}, \cF_{\pi_{\fm}[\fm]})$ is $\overline{\sigma}|_{\mathrm{Gal}_{F^+_{\fp}}}$-typic; and in particular
\item $H^{n-1}_{\emph{\'et}}(\PP_{\C_p}^{n-1}, \cF_{\pi_{\fm}[\fm]})$ determines $\overline{\sigma}|_{\mathrm{Gal}_{F^+_{\fp}}}$ up to isomorphism.
\end{itemize}
\end{thm}

\begin{rmk}
As mentioned in Remark~7.6 of \cite{Liu21}, Condition~\ref{cond: main} can be reduced to a result on the Gelfand--Kirillov dimension of $\pi_{\fm}[\fm]$ using Theorem~B of \cite{GN}. Under standard Taylor--Wiles conditions and mild genericity on $\overline{\sigma}|_{\mathrm{Gal}_{F^+_{\fp}}}^{\rm{ss}}$, the Gelfand-Kirillov dimension of $\pi_{\fm}[\fm]$ is known when $n=2$ and $F^+_{\fp}$ is unramified due to \cite{BHHMS20} and \cite{HW20}. However, under the same assumption in \cite{BHHMS20} and \cite{HW20}, we already know that $\pi_{\fm}[\fm]$ determines $\overline{\sigma}|_{\mathrm{Gal}_{F^+_{\fp}}}$ thanks to \cite{BD14}, whose proof is significantly simpler than that of \cite{BHHMS20} and \cite{HW20}. One expects the determination of the Gelfand-Kirillov dimension of $\pi_{\fm}[\fm]$ for general $n$ and $F^+_{\fp}$ to be a difficult problem, and so is the flatness in Condition~\ref{cond: main}. Concerning the alternative approach generalizing \cite{BD14} (without assuming Condition~\ref{cond: main}), \cite{LLMPQ} shows that $\pi_{\fm}[\fm]$ determines $\overline{\sigma}|_{\mathrm{Gal}_{F^+_{\fp}}}$ when $F^+_{\fp}$ is unramified and $\overline{\sigma}|_{\mathrm{Gal}_{F^+_{\fp}}}$ is Fontaine--Laffaille (assuming standard Taylor--Wiles conditions and mild genericity on $\overline{\sigma}|_{\mathrm{Gal}_{F^+_{\fp}}}^{\rm{ss}}$).
\end{rmk}

\begin{rmk}\label{rmk: non flat}
We write $V_1^\star$ for the image of
\begin{equation}\label{equ: std map}
H^{n-1}_{\text{\'et}}(\PP_{\C_p}^{n-1}, \cF_{\pi_{\fm}[\fm]})\rightarrow H^{n-1}_{\text{\'et}}(\PP_{\C_p}^{n-1}, \cF_{\pi_{\fm}})[\fm]
\end{equation}
and consider the following condition
\begin{equation}\label{equ: JH factor}
\mathrm{JH}_{\F[\mathrm{Gal}_{F^+_{\fp}}]}(V_1^\star)=\mathrm{JH}_{\F[\mathrm{Gal}_{F^+_{\fp}}]}(\overline{\sigma}|_{\mathrm{Gal}_{F^+_{\fp}}}).
\end{equation}
Then we have the following observations.
\begin{enumerate}[label=(\roman*)]
\item Assuming Condition~\ref{cond: middle deg}, Condition~\ref{cond: mult free} and (\ref{equ: JH factor}),
we can deduce that $V_1^\star$ determines $\overline{\sigma}|_{\mathrm{Gal}_{F^+_{\fp}}}$ up to isomorphism from Proposition~\ref{prop: general capture}.
However, one needs to be careful that $V_1^\star$ \emph{a priori} depends on the structure of $\pi_{\fm}$ rather than $\pi_{\fm}[\fm]$, and thus the result above for $V_1^\star$ is not sufficient to imply that $\pi_{\fm}[\fm]$ determines $\overline{\sigma}|_{\mathrm{Gal}_{F^+_{\fp}}}$ up to isomorphism. Suppose that (\ref{equ: std map}) is indeed an embedding, then $V_1^\star\cong H^{n-1}_{\text{\'et}}(\PP_{\C_p}^{n-1}, \cF_{\pi_{\fm}[\fm]})$ and thus Proposition~\ref{prop: general capture} gives an alternative approach to Theorem~\ref{thm: main} without showing that $H^{n-1}_{\text{\'et}}(\PP_{\C_p}^{n-1}, \cF_{\pi_{\fm}[\fm]})$ is $\overline{\sigma}|_{\mathrm{Gal}_{F^+_{\fp}}}$-typic.
\item There are examples in \cite{HW21} (with $n=2$) such that
\begin{itemize}
\item Condition~\ref{cond: middle deg}, Condition~\ref{cond: mult free} and (\ref{equ: JH factor}) hold but Condition~\ref{cond: main} fails;
\item the map (\ref{equ: std map}) is an embedding; and
\item $V_1^\star\cong H^{n-1}_{\text{\'et}}(\PP_{\C_p}^{n-1}, \cF_{\pi_{\fm}[\fm]})$ is not $\overline{\sigma}|_{\mathrm{Gal}_{F^+_{\fp}}}$-typic.
\end{itemize}
\item When $n\geq 3$, we do not know how to prove (\ref{equ: JH factor}) or to prove that (\ref{equ: std map}) is an embedding without using Condition~\ref{cond: main}.
\end{enumerate}
\end{rmk}


\begin{rmk}
Let $\overline{\chi}_1,\overline{\chi}_2$ be two distinct characters $\mathrm{Gal}_{F^+_{\fp}}\rightarrow \F^\times$, $r_1, r_2\geq 2$ be two integers and assume that $\overline{\sigma}|_{\mathrm{Gal}_{F^+_{\fp}}}\cong \overline{\chi}_1^{\oplus r_1}\oplus\overline{\chi}_2^{\oplus r_2}$ which is not multiplicity free.
Then for each infinite dimensional $\F$-space $W$ with trivial $\mathrm{Gal}_{F^+_{\fp}}$-action, the isomorphism class of the $\F[\mathrm{Gal}_{F^+_{\fp}}]$-module $W\otimes_{\F}\overline{\sigma}|_{\mathrm{Gal}_{F^+_{\fp}}}$ does not depend on the choice of $r_1$ and $r_2$. In particular, the $\F[\mathrm{Gal}_{F^+_{\fp}}]$-module $H^{n-1}_{\text{\'et}}(\PP_{\C_p}^{n-1}, \cF_{\pi_{\fm}})[\fm]$ cannot determine $r_1$ and $r_2$.
In order to prove $\pi_{\fm}[\fm]$ determines $\overline{\sigma}|_{\mathrm{Gal}_{F^+_{\fp}}}$ for such $\overline{\sigma}|_{\mathrm{Gal}_{F^+_{\fp}}}$, we expect the $D^\times$-action on $H^{n-1}_{\text{\'et}}(\PP_{\C_p}^{n-1}, \cF_{\pi_{\fm}[\fm]})$ to be essential.
\end{rmk}

\end{document}